\documentclass[a4paper,11pt]{amsart}

\usepackage{amsmath,amssymb,amsfonts,latexsym}

\newtheorem{theorem}{Theorem}[section]
\newtheorem{lemma}[theorem]{Lemma}
\newtheorem{corollary}[theorem]{Corollary}

\theoremstyle{definition}

\theoremstyle{remark}

\newlength\longest

\newcommand{\R}{{\mathbb R}}

\begin{document}

\title[Generalized H\"{o}lder's Inequality in Orlicz Spaces]{Generalized H\"{o}lder's Inequality in Orlicz Spaces}

\author{Ifronika${}^{1}$}
\address{${}^{1}$Analysis and Geometry Group, Faculty of Mathematics and Natural Sciences, Bandung Institute of Technology, Jl. Ganesha 10, Bandung 40132, Indonesia}
\email{ifronika@itb.ac.id}

\author{A.A. Masta${}^{2}$}
\address{${}^{2}$Department of Mathematics Education, Universitas Pendidikan Indonesia, Jl. Dr. Setiabudi 229, Bandung 40154, Indonesia.}

\email{alazhari.masta@upi.edu}

\author{M. Nur${}^{3,A}$}
\address{${}^{3}$Analysis and Geometry Group, Faculty of Mathematics and Natural Sciences, Bandung Institute of Technology, Jl. Ganesha 10, Bandung 40132, Indonesia}

\address{${}^{A}$\emph{Permanent Address}: Department of Mathematics, Hasanuddin University, Jl. Perintis Kemerdekaan KM 10, Makassar 90245, Indonesia}

\email{muhammadnur@unhas.ac.id}

\author{H. Gunawan${}^{4}$}
\address{${}^{4}$Analysis and Geometry Group, Faculty of Mathematics and Natural Sciences, Bandung Institute of Technology, Jl. Ganesha 10, Bandung 40132, Indonesia}
\email{hgunawan@itb.ac.id}


\begin{abstract}

Orlicz spaces are generalizations of Lebesgue spaces. The sufficient and necessary conditions for generalized H\"{o}lder's inequality in Lebesgue spaces and in weak Lebesgue spaces are well known.  The aim of this paper is to present sufficient and necessary conditions for generalized
H\"{o}lder's inequality in Orlicz spaces and in weak Orlicz spaces, which are obtained through estimates for characteristic functions of balls in $\R^n$.

\vspace{2mm}

\noindent\textsc{2010 Mathematics Subject Classification.} Primary  26D15; Secondary 46B25, 46E30.

\vspace{2mm}

\noindent\textsc{Keywords and phrases.} Generalized H\"{o}lder's inequality, Orlicz spaces, Weak Orlicz spaces.

\end{abstract}

\thanks{The first and fourth authors are supported by P3MI Program 2018. The second author is supported by Hibah Disertasi Doktor 2018.}


\maketitle


\section {Introduction and Preliminaries}


Orlicz spaces are generalizations of Lebesgue spaces which were introduced by Z.W. Birnbaum and W. Orlicz in 1931
\cite{Orlicz}. Let us first recall the definition of Young function, Orlicz spaces, and weak Orlicz spaces. A function $\Phi:[0,\infty)\to [0,\infty)$ is called \textit{a Young function} if $\Phi$ is
convex, left-continuous, $\Phi(0) = 0$, and $\lim \limits_{t\to\infty}
\Phi(t) = \infty$. 

Let $\Phi$ be a Young function, we define \textit{the Orlicz
space} $L_\Phi(\R^n)$ to be
the set of measurable functions $f : \R^n \rightarrow \mathbb{R}$ such that
$$ \int_{\R^n} \Phi ( a|f(x)|) dx < \infty$$
for some $a > 0$.  The Orlicz space $L_\Phi(\R^n)$ is a Banach space with respect to the norm
$$
\| f \|_{L_\Phi(\R^n)} := \inf \left\{  {b>0:
	\int_{\R^n}\Phi \left(\frac{|f(x)|}{b} \right) dx \leq1}\right\}
$$
(see \cite{Christian, Luxemburg}). Note that, if we take an arbitrary $f \in L_\Phi(\R^n)$, then there exists $b >0$ such that $\int_{\R^n} \Phi \Bigl(\frac{|f(x)|}{b}\Bigr) dx \leq 1$. If $\Phi(t) := t^p$ for some $ 1 \leq p < \infty$ then $L_\Phi(\R^n)=L^p(\mathbb{R}^n)$. Thus, the Orlicz space $L_\Phi(\R^n)$ can be viewed as a generalization of Lebesgue space $L^p(\mathbb{R}^n)$.

On the other hand, for $\Phi$ is a Young function, \textit{the weak Orlicz space} $wL_{\Phi}(\mathbb{R}^n)$ is the set of measurable functions $f : \mathbb{R}^n \rightarrow \mathbb{R} $ such
that
$$
\| f \|_{wL_\Phi(\mathbb{R}^n)} := \inf \left\{  {b>0:
	\mathop {\sup }\limits_{t > 0} \Phi(t) \Bigl| \Bigl\{ x \in \mathbb{R}^n : \frac{|f(x)|}{b} > t \Bigr\}
	\Bigr| \leq1}\right\} < \infty.
$$

\noindent{\tt Remark}. Note that $\| \cdot \|_{wL_\Phi(\mathbb{R}^n)}$ defines a quasi-norm in $wL_{\Phi}(\mathbb{R}^n)$, and that $(wL_{\Phi}(\mathbb{R}^n),\| \cdot \|_{wL_\Phi(\mathbb{R}^n)})$ forms a quasi-Banach space (see \cite{Bekjan, Yong}).

The relation between Orlicz spaces and weak Orlicz spaces is clear, as
presented in the following theorem.
\medskip

\begin{theorem}\label{theorem:1}\cite{Yong}
	Let $\Phi$ be a Young function. Then $L_{\Phi}(\mathbb{R}^n) \subset wL_{\Phi}(\mathbb{R}^n)$
	with $$\| f \|_{wL_\Phi(\mathbb{R}^n)} \leq \| f \|_{L_\Phi(\mathbb{R}^n)},$$for every $f\in L_\Phi(\mathbb{R}^n)$.
\end{theorem}

The study of Orlicz spaces and weak Orlicz spaces were widely investigated during last decades, see \cite{Kufner,Ning,Luxemburg,Lech,Masta1,Oneil}.  In 1965, O'Neil \cite{Oneil} obtained sufficient and necessary conditions
for the H\"{o}lder's inequality in Orlicz spaces, as in the following theorem.\\

\begin{theorem}\label{theorem:1.1}
	Let $\Phi_i$ be Young functions and $ \Phi_{i}^{-1}(s)=\inf \{r \geq 0 :
	\Phi_i (r) > s \}$ for $i=1,2,3$.
	Then the following statements are equivalent:
	
	{\parindent=0cm
		{\rm (1)} There exists a constant $C>0$ such that for all $ t \geq 0$ we have $$\Phi_{1}^{-1}(t)\Phi_{2}^{-1}(t) \leq C \Phi_{3}^{-1}(t).$$  
		
		{\rm (2)} There exists a constant $C>0$ such that for all $s, t \ge 0$, $$ \Phi_{3} \Bigl(\frac{st}{C}\Bigr) \leq \Phi_1(s) +\Phi_2(t).$$

		{\rm (3)} There exists a constant $M>0$ such that $$\| f g \|_{L_{\Phi_3}(\R^n)} \leq M\| f\|_{L_{\Phi_1}(\R^n)}
		\| g\|_{L_{\Phi_2}(\R^n)}$$  for every $f \in L_{\Phi_1}(\R^n)$
		and $g \in L_{\Phi_2}(\R^n)$.
		
		{\rm (4)} For every $f \in L_{\Phi_1}(\R^n)$
		and $g \in L_{\Phi_2}(\R^n)$, then $fg \in L_{\Phi_3}(\R^n)$. 
		\par}
\end{theorem}

\medskip

In 2016, Masta \textit{et al.} \cite{Masta2} obtained sufficient and necessary conditions
for the generalized H\"{o}lder's inequality in Lebesgue spaces. Related result about sufficient and necessary conditions
for the generalized H\"{o}lder's inequality can be found in \cite{Ifro}. 


Motivated by these results, the purpose of this study is to get the sufficient and necessary conditions for the generalized H\"{o}lder's inequality in Orlicz spaces and extend the results to weak Orlicz spaces.

The rest of this paper is organized as follows. The main results are
presented in Sections 2. In Section 2, we state the sufficient and necessary conditions for generalized H\"{o}lder's inequality in Orlicz spaces as Theorem
\ref{theorem:2.1}. An analogous result for the weak
Orlicz spaces is stated as Theorem \ref{theorem:3.1}.

\bigskip

To prove our results, we pay attention to the characteristic functions
of balls in $\mathbb{R}^n$ and the following lemmas.

\medskip
\begin{lemma}\label{lemma:2.2} \cite{Christian}
	Let $\Phi$ be a Young function and $f \in L_{\Phi}(\mathbb{R}^n)$. If $0 <
	\| f \|_{L_\Phi(\mathbb{R}^n)} <\infty$, then $\int_{\mathbb{R}^n}\Phi\left(
	\frac{|f(x)|}{\| f \|_{L_\Phi(\mathbb{R}^n)}} \right)  dx \leq 1 $. Furthermore,
	$\| f \|_{L_\Phi(\mathbb{R}^n)} \leq 1 $ if only if $\int_{\mathbb{R}^n}\Phi(|f(x)|) dx
	\leq 1.$
\end{lemma}
\bigskip

\begin{lemma}\label{lemma:1.1}\cite{Masta1}
	Let $\Phi$ be a Young function. If $ \Phi^{-1}(s):=\inf \{r \geq 0 :
	\Phi (r) > s \}$, then we have the following properties:
	
	{\parindent=0cm
		{\rm (1)} $\Phi^{-1}(0) = 0$.
		
		{\rm (2)} $ \Phi^{-1}(s_1) \leq \Phi^{-1}(s_2)$ for  $s_1 \leq s_2$.
		
		{\rm (3)} $\Phi (\Phi^{-1}(s)) \leq s \leq \Phi^{-1}(\Phi(s))$ for $0 \leq s <
		\infty$.
		\par}
\end{lemma}

\bigskip

\begin{lemma}\label{lemma:2.4}\cite{Yong,Masta1} Let $\Phi$ be a Young function,
	$a\in\mathbb{R}^n$, and $r>0$. Then
	$$ \| \chi_{B(a,r)} \|_{L_\Phi(\mathbb{R}^n)}= \frac{1}{\Phi^{-1}(\frac{1}
		{|B(a,r)|})},$$ where $|B(a,r)|$ denotes the volume of open ball $B(a,r)$.
\end{lemma}


\section {Results and Discussion}


The following theorem presents sufficient and necessary conditions for generalized H\"{o}lder's inequality in Orlicz spaces.

\bigskip

\begin{theorem}\label{theorem:2.1} 
	Let $m\geq 2$. If $\Phi$ and $\Phi_i$ are Young functions for $i=1,\dots, m$,
	then the following statements are equivalent:
	
	{\parindent=0cm
		{\rm (1)} There exists a constant $C>0$ such that $$ \prod\limits_{i=1}^m \Phi_{i}^{-1}(t) \leq C \Phi^{-1}(t)$$ for every $ t \geq 0$.  
		
		{\rm (2)} There exists a constant $C>0$ such that for all $t_i \ge 0$, $i = 1,\cdots,m$, $$ \Phi \left( \frac{\prod\limits_{i=1}^m t_i}{C}\right)  \leq \sum_{i=1}^{m} \Phi_i(t_i).$$
		
		{\rm (3)} There exists a constant $M>0$ such that $$\left\Vert \prod\limits_{i=1}^m f_i \right\Vert_{L_{\Phi}(\mathbb{R}^n)}\leq M \prod\limits_{i=1}^m
		\| f_i \|_{L_{\Phi_i}(\mathbb{R}^n)},$$ for every $f_i\in L_{\Phi_i}(\mathbb{R}^n)$,
		$i=1,\dots,m$.
		
		{\rm (4)} For every $f_i \in L_{\Phi_i}(\mathbb{R}^n)$, then $\prod\limits_{i=1}^m f_i \in L_{\Phi}(\mathbb{R}^n)$.
		
		\par}
\end{theorem}

\medskip
\begin{proof}
	($(1) \Rightarrow (2)$) Suppose that (1) hold. Since $\Phi$ is a Young function and using Lemma \ref{lemma:1.1} we have $$t_i \leq \Phi_{i}^{-1}\Bigl(\Phi_{i}(t_i)\Bigr)\leq \Phi_{i}^{-1}\Bigl(\sum\limits_{i=1}^{m}\Phi_i(t_i)\Bigr)$$ for $i = 1,\cdots,m$. Hence $$\prod\limits_{i=1}^{m} t_i\leq \prod\limits_{i=1}^{m}\Phi_{i}^{-1}\Bigl(\sum\limits_{i=1}^{m}\Phi_i(t_i)\Bigr) \leq C \Phi^{-1}\Bigl(\sum\limits_{i=1}^{m}\Phi_i(t_i)\Bigr).$$ 
	
Because $\Phi$ is increasing and by Lemma \ref{lemma:1.1} (3), we have $$\Phi\Bigl(\frac{1}{C}\prod\limits_{i=1}^{m} t_i \Bigr) \leq \Phi\Bigl(\Phi^{-1}\Bigl(\sum\limits_{i=1}^{m}\Phi_i(t_i)\Bigr)\Bigr) \leq \sum\limits_{i=1}^{m}\Phi_i(t_i).$$
	
($(2) \Rightarrow (3)$). Suppose that (2) hold. Let $f_i$ be an element of $L_{\Phi_i}(\mathbb{R}^n)$. By Lemma \ref{lemma:2.2}, we have $$\int_{\mathbb{R}^n}\Phi_{i} \Bigl(\frac{|f_{i}(x)|}{\|f_i\|_{L_{\Phi_i}(\mathbb{R}^n)}} \Bigr) dx \leq1,$$ 
	for every $i=1,\dots,m$. On the other hand, we have
	\begin{align*}
	\int_{\mathbb{R}^n}\Phi\Bigl( \frac{1}{mC}\prod\limits_{i=1}^{m} \frac{|f_{i}(x)|}{\|f_i\|_{L_{\Phi_i}(\mathbb{R}^n)}}\Bigr) dx &\leq \frac{1}{m} \int_{\mathbb{R}^n}\Phi\Bigl( \frac{1}{C}\prod\limits_{i=1}^{m} \frac{|f_{i}(x)|}{\|f_i\|_{L_{\Phi_i}(\mathbb{R}^n)}}\Bigr) dx \\
	&\leq \frac{1}{m} \sum\limits_{i=1}^{m} \int_{\mathbb{R}^n}\Phi_{i}\Bigl(\frac{|f_{i}(x)|}{\|f_{i}\|_{L_{\Phi_i}(\mathbb{R}^n)}}\Bigr) dx \leq 1.
	\end{align*}
	
	By definition of $\| \cdot \|_{L_{\Phi}(\mathbb{R}^n)}$, we conclude that $$\left\|\prod\limits_{i=1}^{m} f_{i}\right\|_{L_{\Phi}(\mathbb{R}^n)} \leq mC \prod\limits_{i=1}^{m} \|f_{i}\|_{L_{\Phi_i}(\mathbb{R}^n)}.$$

($(3) \Leftrightarrow (4)$). Next, it is easy to prove that (3) implies (4). Now, suppose that (4) holds, then there exists $ \alpha >0$ such that $$\int_{\mathbb{R}^n}\Phi \Biggl(\frac{\prod\limits_{i=1}^{m}|f_{i}(x)|}{\alpha} \Biggr) dx \leq1.$$  
	
	By setting $M := \frac{\alpha}{\prod\limits_{i=1}^{m}\|f_i\|_{L_{\Phi_i}(\mathbb{R}^n)}} >0$, we have 
	$$\int_{\mathbb{R}^n}\Phi \Biggl(\frac{\prod\limits_{i=1}^{m}|f_{i}(x)|}{M\prod\limits_{i=1}^{m}\|f_i\|_{L_{\Phi_i}(\mathbb{R}^n)}} \Biggr) dx = \int_{\mathbb{R}^n}\Phi \Biggl(\frac{\prod\limits_{i=1}^{m}|f_{i}(x)|}{\alpha} \Biggr) dx \leq1.$$ 	
	
	By definition of Orlicz-norm we have $\left\Vert \prod\limits_{i=1}^m f_i \right\Vert_{L_{\Phi}(\mathbb{R}^n)}\leq M \prod\limits_{i=1}^m
	\| f_i \|_{L_{\Phi_i}(\mathbb{R}^n)}$.
	\\
($(3)\Rightarrow (1)$). Suppose that (3)  holds. Take an arbitrary open ball $B_0:=B(0,r_0)$ for $r_0>0$. Observe that $\|\chi_{B_0}\|_{L_{\Phi}(\mathbb{R}^n)} 
	= \left\|\prod\limits_{i=1}^{m}\chi_{B_0}\right\|_{L_{\Phi}(\mathbb{R}^n)}$. By using Lemma \ref{lemma:2.4} we have
$$\frac{1}{\Phi^{-1} \Big(\frac{1}{|B_0|} \Bigr)} =\|\chi_{B_0}\|_{L_{\Phi}(\mathbb{R}^n)} \leq M \prod\limits_{i=1}^{m} \|\chi_{B_0}\|_{L_{\Phi_i}(\mathbb{R}^n)} = M \prod\limits_{i=1}^{m}\frac{1}{\Phi^{-1}_i \Big(\frac{1}{|B_0|} \Bigr)}$$
for every open ball $B_0 \subseteq \mathbb{R}^n$. Since $r_0>0$ is arbitrary, we get $$\prod\limits_{i=1}^{m}\Phi_{i}^{-1}(t)\leq C \Phi^{-1}(t)$$ for every $t \geq 0$.  		
\end{proof}
\medskip

\noindent{\tt Remark}. For $m=2$ Theorem \ref{theorem:2.1} reduces to Theorem \ref{theorem:1.1}. Note that, for $ m=1$ Theorem \ref{theorem:2.1} may be viewed as inclusion properties of Orlicz spaces in \cite{Lech,Masta1}.
\medskip

\begin{corollary}\label{corollary:2.1}
	Let $m\geq 2$. If $1 \leq p, p_i < \infty$ for $i=1,\dots, m$,
	then the following statements are equivalent:
	
	{\parindent=0cm
		{\rm (1)}  $\sum_{i=1}^{m} \frac{1}{p_i} = \frac{1}{p}$.  
		
		{\rm (2)} $\left\Vert \prod\limits_{i=1}^m f_i \right\Vert_{L^{p}(\mathbb{R}^n)}\leq \prod\limits_{i=1}^m
		\| f_i \|_{L^{p_i}(\mathbb{R}^n)},$ for every $f_i\in L^{p_i}(\mathbb{R}^n)$,
		$i=1,\dots,m$.
		
		{\rm (3)} For every $f_i \in L^{p_i}(\mathbb{R}^n)$, then $\prod\limits_{i=1}^m f_i \in L^{p}(\mathbb{R}^n)$.
		
		\par}
\end{corollary}
\begin{proof}
	The proof that $(1)$ and $(2)$ are equivalent can be found in \cite{Ifro}. Next, by setting $\Phi(t):=t^p$ and $\Phi_i(t):=t^{p_i}$  for $ 1\leq p, p_i < \infty$ in the Theorem \ref{theorem:2.1}, we have $ (2)$ and $(3)$ are equivalent.
\end{proof}
For weak Orlicz spaces, we also have the sufficient and necessary conditions for generalized H\"{o}lder's inequality. To prove the result, we use the following lemmas.
\bigskip
\begin{lemma}\label{lemma:1.6}
	If $ f \in wL_{\Phi}(\mathbb{R}^n)$, then $$ \mathop {\sup }\limits_{t > 0} \Phi(t)\Bigl| \Bigl \{ x \in \mathbb{R}^n : \frac{|f(x)|}{\| f \|_{wL_{\Phi}(\mathbb{R}^n)} + \epsilon } > t \Bigr \} \Bigr| \leq1$$ for every $\epsilon > 0$.

\end{lemma}

\begin{proof}
Let $f\in wL_{\Phi} (\mathbb{R}^{n})$. Take an arbitrary $\epsilon > 0$, then there exists $b_{\epsilon} > 0$ such that $b_{\epsilon} \leq \| f \|_{wL_{\Phi}(\mathbb{R}^n)} + \epsilon$ and $$ \mathop {\sup }\limits_{t > 0}\Phi(t)\Bigl| \Bigl\{ x \in \mathbb{R}^n : \frac{|f(x)|}{b_{\epsilon} } > t \Bigr\} \Bigr|\leq1.$$

Since $\frac{|f(x)|}{b_{\epsilon}} \geq \frac{|f(x)|}{\| f \|_{wL_{\Phi}(\mathbb{R}^n)} + \epsilon}$, we have

$$\Phi(t)\Bigl|\Bigl\{ x \in \mathbb{R}^n : \frac{|f(x)|}{\| f \|_{wL_{\Phi}(\mathbb{R}^n)}+\epsilon} > t \Bigr\} \Bigr| \leq \Phi(t)\Bigl| \Bigl\{ x \in \mathbb{R}^n : \frac{|f(x)|}{b_{\epsilon} } > t \Bigr\}\Bigl| \leq 1$$ for every $t > 0$.

By taking supremum over $t>0$, we conlude that $$ \mathop {\sup }\limits_{t > 0} \Phi(t) \Bigl| \Bigl\{ x \in \mathbb{R}^n : \frac{|f(x)|}{\|f\|_{wL_{\Phi}(\mathbb{R}^n)} + \epsilon } > t \Bigr\} \Bigr| \leq1 $$ for every $\epsilon > 0$.
\end{proof}
\medskip

\begin{lemma}\label{lemma:1.7}\cite{Yong,Masta1} Let $\Phi$ be a Young function,
	$a\in\mathbb{R}^n$, and $r>0$. Then
	$$\| \chi_{B(a,r)} \|_{wL_\Phi(\mathbb{R}^n)} = \frac{1}{\Phi^{-1}(\frac{1}
		{|B(a,r)|})},$$ where $|B(a,r)|$ denotes the volume of $B(a,r)$.
\end{lemma}

\begin{lemma}\label{lemma:1.8}
	If $ f \in wL_\Phi(\mathbb{R}^n)$, then there exists $ \alpha >0$ such that $$\mathop {\sup }\limits_{t > 0} \Phi(t)\left|  \left\lbrace  x \in \mathbb{R}^n : \frac{|f(x)|}{\alpha} > t \right\rbrace  \right| \leq1.$$ 
\end{lemma}

We leave the proof of Lemma \ref{lemma:1.8} to the reader. Finaly, we come to the generalized H\"{o}lder's inequality in weak Orlicz spaces as follows.

\begin{theorem}\label{theorem:3.1}
	Let $m\geq 2$. If $\Phi$ and $\Phi_i$ are Young functions for $i=1,\dots, m$,
	then the following statements are equivalent:
	
	{\parindent=0cm
		{\rm (1)} There exists a constant $C>0$ such that $$ \prod\limits_{i=1}^m \Phi_{i}^{-1}(t) \leq C \Phi^{-1}(t)$$ for every $ t \geq 0$.  
		
		{\rm (2)} There exists a constant $C >0$ such that for all $t_i \ge 0$, $$ \Phi \left( \frac{\prod\limits_{i=1}^m t_i}{C}\right)  \leq \sum_{i=1}^{m} \Phi_i(t_i).$$
		
		{\rm (3)} There exists a constant $M>0$ such that $$\left\Vert \prod\limits_{i=1}^m f_i \right\Vert_{wL_{\Phi}(\mathbb{R}^n)}\leq M \prod\limits_{i=1}^m
		\| f_i \|_{wL_{\Phi_i}(\mathbb{R}^n)},$$ for every $f_i\in wL_{\Phi_i}(\mathbb{R}^n)$,
		$i=1,\dots,m$.
		
		{\rm (4)} For every $f_i \in wL_{\Phi_i}(\mathbb{R}^n)$, then $\prod\limits_{i=1}^m f_i \in wL_{\Phi}(\mathbb{R}^n)$.
		
		\par}
\end{theorem}

\begin{proof}
	As before, we have that (1) and (2) are equivalent. We shall prove that (2) implies (3) and (3) implies (1). Suppose that (2) hold. Let $f_i \in wL_{\Phi_i}(\mathbb{R}^n)$. By Lemma \ref{lemma:1.6}, we have $$\Phi_{i} (t)\Bigl| \Bigl\{ x \in \mathbb{R}^n : \frac{|f_i(x)|}{\| f_i \|_{wL_{\Phi_i}(\mathbb{R}^n)} + \frac{\| f_i \|_{wL_{\Phi_i}(\mathbb{R}^n)}}{k}} > t \Bigr\} \Bigr| \leq1,$$ for every $ k \in \mathbb{N}$.

	Define $A_{\Phi}(t):=\Phi(t)\left|  \left\lbrace  x \in \mathbb{R}^n : \frac{\prod\limits_{i=1}^{m}|f_{i}(x)|}{mC\prod\limits_{i=1}^{m}(1+\frac{1}{k})\|f_i\|_{wL_{\Phi_i}(\mathbb{R}^n)}} > t \right\rbrace  \right|$. 
	
	Next by setting $ t_0 :=  \frac{t mC\prod\limits_{i=1}^{m}(1+\frac{1}{k})\|f_i\|_{wL_{\Phi_i}(\mathbb{R}^n)}}{\prod\limits_{i=1}^{m}|f_{i}(x)|}$, we have
	\begin{align*}
	A_{\Phi}(t)& =  \Phi \left( \frac{\prod\limits_{i=1}^{m}(t_0)^{\frac{1}{m}}|f_{i}(x)|}{mC\prod\limits_{i=1}^{m}(1+\frac{1}{k})\|f_i\|_{wL_{\Phi_i}(\mathbb{R}^n)}}\right)|\{ x \in \mathbb{R}^n : 1 > t_0 \}|  \\
	& \leq \frac{1}{m}\Phi \left( \frac{\prod\limits_{i=1}^{m}(t_0)^{\frac{1}{m}}|f_{i}(x)|}{C\prod\limits_{i=1}^{m}(1+\frac{1}{k})\|f_i\|_{wL_{\Phi_i}(\mathbb{R}^n)}}\right)|\{ x \in \mathbb{R}^n : 1 > t_0 \}|  \\
	& \leq  \frac{1}{m} \left(\sum\limits_{i=1}^{m}  \Phi_{i}\Bigl(\frac{(t_0)^{\frac{1}{m}}|f_{i}(x)|}{(1+\frac{1}{k})\|f_{i}\|_{wL_{\Phi_i}(\mathbb{R}^n)}}\Bigr)\right)|\{ x \in \mathbb{R}^n : 1 > t_0 \}| .\\
	\end{align*}
	
	On the other hand,

	\begin{align*}
	\Phi_{i}\Bigl(\frac{(t_0)^{\frac{1}{m}}|f_{i}(x)|}{(1+\frac{1}{k})\|f_{i}\|_{wL_{\Phi_i}(\mathbb{R}^n)}}\Bigr)&| \{ x \in \mathbb{R}^n : 1 > t_0 \}|\\
	& = \Phi_{i}(t_i)\Bigl| \Bigl\{ x \in \mathbb{R}^n : \Bigl(\frac{|f_i(x)|}{(1+\frac{1}{k})\|f_i\|_{wL_{\Phi_i}(\mathbb{R}^n)}} \Bigr)^m > t_i^m \Bigr\} \Bigr| \\
	& = \Phi_{i}(t_i)\Bigl| \Bigl\{ x \in \mathbb{R}^n : \frac{|f_i(x)|}{(1+\frac{1}{k})\|f_i\|_{wL_{\Phi_i}(\mathbb{R}^n)}} > t_i \Bigr\}\Bigr| \\
	& \leq 1,
	\end{align*}
	
	where $ t_i := \frac{t_0^{\frac{1}{m}}|f_i(x)|}{(1+\frac{1}{k})\|f_i\|_{wL_{\Phi_i}(\mathbb{R}^n)}}$ for $i = 1, \cdots, m$.\\
	
	This show that 
	
	\begin{align*}
	A_{\Phi}(t)& =  \Phi \left( \frac{\prod\limits_{i=1}^{m}(t_0)^{\frac{1}{m}}|f_{i}(x)|}{mC\prod\limits_{i=1}^{m}(1+\frac{1}{k})\|f_i\|_{wL_{\Phi_i}(\mathbb{R}^n)}}\right)|\{ x \in \mathbb{R}^n : 1 > t_0 \}|  \\
	& \leq \frac{1}{m} \Phi \left( \frac{\prod\limits_{i=1}^{m}(t_0)^{\frac{1}{m}}|f_{i}(x)|}{C\prod\limits_{i=1}^{m}(1+\frac{1}{k})\|f_i\|_{wL_{\Phi_i}(\mathbb{R}^n)}}\right)|\{ x \in \mathbb{R}^n : 1 > t_0 \}|  \\
	& \leq \frac{1}{m}\left(\sum\limits_{i=1}^{m}  \Phi_{i}\Bigl(\frac{(t_0)^{\frac{1}{m}}|f_{i}(x)|}{(1+\frac{1}{k})\|f_{i}\|_{wL_{\Phi_i}(\mathbb{R}^n)}}\Bigr)\right)|\{ x \in \mathbb{R}^n : 1 > t_0 \}| \\
	&\leq 1.
	\end{align*}
	
	Since $ t > 0$ is an arbitrary positive real number, we get $$ \mathop {\sup }\limits_{t > 0} \Phi(t)\left|  \left\lbrace  x \in \mathbb{R}^n : \frac{\prod\limits_{i=1}^{m}|f_{i}(x)|}{mC\prod\limits_{i=1}^{m}(1+\frac{1}{k})\|f_i\|_{wL_{\Phi_i}(\mathbb{R}^n)}} > t \right\rbrace  \right| \leq1.$$
	
	By definition of $\| \cdot \|_{ wL_{\Phi}(\mathbb{R}^n)}$, we have $$\left\|\prod\limits_{i=1}^{m} f_{i}\right\|_{ wL_{\Phi}(\mathbb{R}^n)} \leq mC (1+\frac{1}{k})^{m}\prod\limits_{i=1}^{m} \|f_{i}\|_{wL_{\Phi_i}(\mathbb{R}^n)}.$$
	
	For $ k \rightarrow \infty$, we have $(1+\frac{1}{k})^{m} \rightarrow 1$. Hence we conclude that  $$\left\|\prod\limits_{i=1}^{m} f_{i}\right\|_{ wL_{\Phi}(\mathbb{R}^n)} \leq mC\prod\limits_{i=1}^{m} \|f_{i}\|_{wL_{\Phi_i}(\mathbb{R}^n)}.$$
	
($(3) \Leftrightarrow (4)$). Next, it is easy to see that (3) implies (4). Now, supppose that (4) holds, by using Lemma \ref{lemma:1.8}, there exists $ \alpha >0$ such that $$\mathop {\sup }\limits_{t > 0} \Phi(t)\left|  \left\lbrace  x \in \mathbb{R}^n : \frac{\prod\limits_{i=1}^{m}|f_{i}(x)|}{\alpha} > t \right\rbrace  \right| \leq1.$$

By setting $M := \frac{\alpha}{\prod\limits_{i=1}^{m}\|f_i\|_{wL_{\Phi_i}(\mathbb{R}^n)}} >0$, we have 
\begin{footnotesize}
$$\mathop {\sup }\limits_{t > 0} \Phi(t)\left|  \left\lbrace  x \in \mathbb{R}^n : \frac{\prod\limits_{i=1}^{m}|f_{i}(x)|}{M\prod\limits_{i=1}^{m}\|f_i\|_{wL_{\Phi_i}(\mathbb{R}^n)}} > t \right\rbrace  \right| = \mathop {\sup }\limits_{t > 0} \Phi(t)\left|  \left\lbrace  x \in \mathbb{R}^n : \frac{\prod\limits_{i=1}^{m}|f_{i}(x)|}{\alpha} > t \right\rbrace  \right|  \leq1.$$ 
\end{footnotesize}	

By definition of $\|\cdot\|_{wL_{\Phi}(\mathbb{R}^n)}$, we have $\left\Vert \prod\limits_{i=1}^m f_i \right\Vert_{wL_{\Phi}(\mathbb{R}^n)}\leq M \prod\limits_{i=1}^m
\| f_i \|_{wL_{\Phi_i}(\mathbb{R}^n)}$. Thus (2) implies (3).
	
Suppose now that (3)  holds. Take an arbitrary open ball $B_0:=B(0,r_0)$ for $r_0>0$. Observe that $\|\chi_{B_0}\|_{wL_{\Phi}(\mathbb{R}^n)}= \left\|\prod\limits_{i=1}^{m}\chi_{B_0}\right\|_{wL_{\Phi}(\mathbb{R}^n)}$. By using Lemma \ref{lemma:1.7}, we have
	
$\frac{1}{\Phi^{-1} \Big(\frac{1}{|B_0|} \Bigr)}=\|\chi_{B_0}\|_{wL_{\Phi}(\mathbb{R}^n)}
	\leq M \prod\limits_{i=1}^{m} \|\chi_{B_0}\|_{wL_{\Phi_i}(\mathbb{R}^n)}
	= M \prod\limits_{i=1}^{m}\frac{1}{\Phi^{-1}_i \Big(\frac{1}{|B_0|} \Bigr)}$
	
	for every open ball $B_0 \subseteq \mathbb{R}^n$. Since $r_0>0$ is arbitrary, we get $$\prod\limits_{i=1}^{m}\Phi_{i}^{-1}(t)\leq M \Phi^{-1}(t)$$
	for every $t \geq 0$. Hence (3) implies (1), and we are done.  		
\end{proof}

\begin{corollary}\label{corollary:2.2}
	Let $m\geq 2$. If $1 \leq p, p_i < \infty$ for $i=1,\dots, m$,
	then the following statements are equivalent:
	
	{\parindent=0cm
		{\rm (1)}  $\sum_{i=1}^{m} \frac{1}{p_i} = \frac{1}{p}$.  
		
		{\rm (2)} $\left\Vert \prod\limits_{i=1}^m f_i \right\Vert_{wL^{p}(\mathbb{R}^n)}\leq \prod\limits_{i=1}^m
		\| f_i \|_{wL^{p_i}(\mathbb{R}^n)},$ for every $f_i\in wL^{p_i}(\mathbb{R}^n)$,
		$i=1,\dots,m$.
		
		{\rm (3)} For every $f_i \in wL^{p_i}(\mathbb{R}^n)$, then $\prod\limits_{i=1}^m f_i \in wL^{p}(\mathbb{R}^n)$.
		
		\par}
\end{corollary}
\begin{proof}
	The proof that  $(1)$ and $(2)$ can be found in \cite{Ifro}. Next, by setting $\Phi(t) =t^p$ and $\Phi_i(t)=t^{p_i}$  for $ 1\leq p, p_i < \infty$ in the Theorem \ref{theorem:3.1}, we have $ (2)$ and $(3)$ are equivalent.
\end{proof}


\section {Concludings Remarks}

We have shown sufficient and necessary conditions for the generalized H\"{o}lder's inequality in Orlicz spaces and in weak Orlicz spaces.
From Theorems \ref{theorem:2.1} and \ref{theorem:3.1}, we see that both generalized H\"older's inequality
in Orlicz spaces and in weak Orlicz spaces are equivalent to the same condition, namely
$\prod\limits_{i=1}^m \Phi_{i}^{-1}(t) \leq C\Phi^{-1}(t)$.


%
%




\begin{thebibliography}{9}

\bibitem{Avram}
F. Avram and L. Brown, ``A generalized H\"{o}lder's inequality and a generalized Szego theorem'',
\emph{Proc. Amer. Math. Soc.} \textbf{107}-3 (1989), 687--695.

\bibitem{Bekjan}
T.N. Bekjan, Z. Chen, P. Liu, and Y. Jiao, ``Noncommutative weak Orlicz spaces and martingale
inequalities'', \emph{Studia Math.} \textbf{204}-3 (2011), 195--212.

\bibitem{Cheung}
W.S. Cheung, ``Generalizations of H\"{o}lder's inequality'', \emph{Int. J. Math. Math. Sci.}
\textbf{26}-1 (2001), 7--10.

\bibitem{Ifro} Ifronika, M. Idris, A.A. Masta, and H. Gunawan, ``Generalized H\"{o}lder's inequality in Morrey spaces'', 
to appear in Matemati\v{c}ki Vesnik.

\bibitem{Yong}
Y. Jiao, ``Embeddings between weak Orlicz martingale spaces'', \emph{J. Math. Appl.} \textbf{378 }
(2011), 220--229.

\bibitem{Kufner}
A. Kufner, O. John, and S. Fu\"{c}ik, \emph{Function Spaces}, Noordhoff International
Publishing, Czechoslovakia, 1977.

\bibitem{Christian}
C. L\'{e}onard, ``Orlicz spaces'', preprint.
[http://cmap.polytechnique.fr/$\sim$leonard/ papers/orlicz.pdf,
accessed on August 17, 2015.]

\bibitem{Ning}
N. Liu and Y. Ye, ``Weak Orlicz space and its convergence theorems'', \emph{Acta Math.
	Sci. Ser. B} {\bf 30}-5 (2010), 1492--1500.

\bibitem{Luxemburg}
W.A.J. Luxemburg, \emph{Banach Function Spaces}, Thesis, Technische Hogeschool te Delft, 1955.

\bibitem{Lech} L. Maligranda, \emph{Orlicz Spaces and Interpolation}, Departamento de Matem\'{a}tica,
Universidade Estadual de Campinas, 1989.

\bibitem{Masta1}
A.A. Masta, H. Gunawan, and W. Setya-Budhi, ``Inclusion property of Orlicz and weak Orlicz spaces'',
\emph{J. Math. Fund. Sci}. \textbf{48}-3 (2016), 193--203.

\bibitem{Masta2} A.A. Masta, H. Gunawan, and W. Setya-Budhi, ``An inclusion property of Orlicz-Morrey spaces'',
\emph{J. Phys.: Conf. Ser.}, \textbf{893} 012015 (2017), 1--8.


\bibitem{Orlicz}
W. Orlicz, \emph{Linear Functional Analysis (Series in Real Analysis Volume 4)}, World Scientific, Singapore, 1992.

\bibitem {Oneil} R. O'Neil, ``Fractional integration in Orlicz spaces. I.'', \emph{Trans. Amer. Math. Soc.}
\textbf{115} (1965), 300--328.

%

\bibitem{Rao}
M.M. Rao and Z.D. Ren, \emph{Theory of Orlicz spaces, volume 146 of Monographs and
	Textbooks in Pure and Applied Mathematics}, Marcel Dekker, Inc., New York, 1991.

\bibitem{Welland}
R. Welland, ``Inclusion relations among Orlicz spaces'', \emph{Proc. Amer. Math. Soc}.
{\bf 17}-1 (1966), 135--139.

\end{thebibliography}
\end{document}